\documentclass[11pt]{article}
\title{\textbf{\Large Existence of Invariant Measures for Reflected SPDEs}}
\author{Jasdeep Kalsi \footnote{\href{mailto:jasdeep.kalsi@maths.ox.ac.uk}{jasdeep.kalsi@maths.ox.ac.uk}}.\\\\
Mathematical Institute, University of Oxford\\\\}
\date{\today}
\usepackage{amsmath, amssymb, amsthm,bbm}
\usepackage{mathrsfs}
\usepackage{enumerate}
\newtheorem{thm}{Theorem}[section]

\newtheorem{cor}[thm]{Corollary}
\newtheorem{defn}[thm]{Definition}

\newtheorem{prop}[thm]{Proposition}

\usepackage[margin=1in]{geometry}
\usepackage{graphicx}
\usepackage{epstopdf}
\usepackage{hyperref}
\usepackage{amsmath}
\usepackage{amsfonts}
\usepackage{setspace}
\usepackage{subcaption}
\newlength{\bibitemsep}
\newlength{\bibparskip}
\let\oldthebibliography\thebibliography
\renewcommand\thebibliography[1]{%
  \oldthebibliography{#1}%
  \setlength{\parskip}{\bibitemsep}%
  \setlength{\itemsep}{\bibparskip}%
}
\usepackage{titlesec}
\titleformat{\section}[block]{\sffamily\Large\bfseries\filcenter}{\thesection}{1em}{}
\usepackage{fancyhdr}
\usepackage{indentfirst}
\usepackage{eqnarray}
\usepackage{enumerate}
\usepackage{commath}
\usepackage{amssymb}
\usepackage{afterpage}
\usepackage{bbm}
\usepackage{algorithm}
\usepackage{algpseudocode}
\usepackage{float}

\numberwithin{equation}{section}
\makeatletter

\newcommand{\xRightarrow}[2][]{\ext@arrow 0359\Rightarrowfill@{#1}{#2}}
\makeatother
\allowdisplaybreaks[1]
\begin{document}
\maketitle
\vspace{-2em}
\begin{abstract}
In this article, we close a gap in the literature by proving existence of invariant measures for reflected SPDEs with only one reflecting barrier. This is done by arguing that the sequence $(u(t,\cdot))_{t \geq 0}$ is tight in the space of probability measures on continuous functions and invoking the Krylov-Bogolyubov theorem. As we no longer have an a priori bound on our solution as in the two-barrier case, a key aspect of the proof is the derivation of a suitable $L^p$ bound which is uniform in time. 
\end{abstract}

\section{Introduction and Statement of Theorem}

The aim of this paper is to argue existence of invariant measures for reflected SPDEs of the form 

\begin{equation}\label{ref_SHE}
\frac{\partial u}{\partial t}= \Delta u + f(x,u(t,x))+ \sigma(x,u(t,x))\frac{\partial^2 W}{\partial t \partial x} + \eta,
\end{equation}
where $(t,x) \in [0,\infty) \times [0,1]$, $u$ satisfies Dirichlet boundary conditions $u(t,0)=u(t,1)=0$ and the initial initial data $u_0$ lies in the space $C_0((0,1))^+$. The measure $\eta$ is a reflection measure which minimally pushes $u$ upwards to ensure that $u \geq 0$. We assume in this paper that the drift and volatility coefficients, $f, \sigma$, are globally bounded. Such equations were originally studied by Nualart and Pardoux in \cite{NP} and they proved existence and uniqueness for the case where $\sigma$ is constant. Donati-Martin and Pardoux then proved existence for volatility coefficients $\sigma$ which are Lipschitz with linear growth in \cite{DMP}. Finally, Xu and Zhang proved existence \emph{and} uniqueness for the equation where $f$ and $\sigma$ satisfy Lipschitz and linear growth conditions in \cite{Xu}. All of these papers focused on the case where the spatial domain is a finite interval, $[0,1]$, with Dirichlet conditions imposed on the endpoints. Otobe extended the existence theory to the case when the spatial domain is $\mathbb{R}$ in \cite{Otobe}, proving uniqueness for the case when $\sigma$ is constant. Uniqueness has also been shown by Hambly and Kalsi in \cite{Paper} for the equation on an unbounded domain provided that $\sigma$ satisfies a Lipschitz condition, with a Lipschitz coefficient which decays exponentially fast in the spatial variable.

Some interesting properties of the solutions have been proved. In \cite{DalangMuller}, the contact sets for the solutions are studied in the case where the drift, $f$, is zero and the volatility, $\sigma$, is constant. In particular, it is shown that at all positive times, the solution  is equal to zero at at most four points almost surely. In \cite{Zam2}, Zambotti examines the behaviour of the reflection measure in more detail, showing that it is absolutely continuous with respect to Lebesgue measure in the space variable, and also that for each point $x$ in space these densities can be viewed as renormalised local time processes for $(u(t,x))_{t \geq 0}$. Zhang proved the strong Feller property of solutions in \cite{Zhang2}, and together with Xu proved a large deviation result for sequences of solutions to such equations with vanishingly small noise in \cite{Xu}. 

In this paper, we are interested in invariant measures for these equations. There are some results on this topic in the literature. Zambotti proved in \cite{Zam1} that the law of the 3D-Bessel bridge is an invariant measure for the equation when $\sigma$ is constant. Otobe then extended this result to the case where the spatial domain is $\mathbb{R}$ in \cite{Otobe2}, proving that the invariant measure is such that the conditional law in an interval is a 3D Bessel bridge with suitable distributions for the endpoints. For the case when the equation has two reflecting walls, above and below the SPDE solution, existence and uniqueness of invariant measures was proved by Yang and Zhang in \cite{YangZhang}. The proof here relied on the a priori bound on the infinity norm of the solution, which is provided by the obstacles. Recently, Xie has proved that invariant measures for the one barrier case, (\ref{ref_SHE}), are unique when they exist, provided that there exist strictly positive constants $c_1$ and $c_2$ such that $c_1 \leq \sigma(x,u) \leq c_2$ in \cite{Xie}.

To the knowledge of the author, existence of invariant measures in the case where there is only reflection at zero has not been proved in the literature. We close this gap here, under the assumption that the drift and volatility coefficients are bounded. We start by proving an $L^p$ bound for our solution when it has been multiplied by an exponential function which dampens the value backwards in time. This control essentially replaces the a priori bound for the two barrier case in the argument in \cite{YangZhang}. We are then able to prove tightness by uniformly controlling the H\"{o}lder norm of the solution, adapting the arguments of \cite{Dalang2} and \cite{YangZhang} in order to do so. 

Before stating the main theorem of this paper, we recall the definition of a solution to a reflected SPDE. We work on a complete probability space $(\Omega, \mathscr{F}, \mathbb{P})$, with $W$ a space-time white noise on this space. This space is equipped with the filtration generated by $W$, $\mathscr{F}_t^W$, which can be written as
\begin{equation*}
\mathscr{F}^W_t: = \sigma(\{ W(A) \; | \; A \in \mathscr{B}([0,t] \times [0,1] \} ) \vee \mathscr{N},
\end{equation*}
where $\mathscr{N}$ here denotes the $\mathbb{P}$-null sets. We further assume that there exist constants $C_f, C_{\sigma}>0$ such that the drift and volatility coefficients, $f$ and $\sigma$ satisfy the following conditions:

\begin{enumerate}[(I)]
\item For every $u, v \in \mathbb{R}^+$ and every $x \in [0,1]$, $$|f(x,u)-f(x,v)| \leq C_f|u-v|.$$
\item For every $u, v \in \mathbb{R}^+$ and every $x \in [0,1]$, $$|\sigma(x,u)-\sigma(x,v)| \leq C_{\sigma}|u-v|.$$
\item For every $x \in [0,1]$, $u \in \mathbb{R}^+$, $$|f(x,u)| \leq C_{f}.$$
\item For every $x \in [0,1]$, $u \in \mathbb{R}^+$, $$|\sigma(x,u)| \leq C_{\sigma}.$$

\end{enumerate}

\begin{defn}\label{weaksolnref}

We say that the pair $(u, \eta)$ is a solution the SPDE with reflection
\begin{equation*}\label{refSPDE}
\frac{\partial u}{\partial t}= \Delta u + f(x, u(t,x))+ \sigma(x,u(t,x)) \frac{\partial^2 W}{\partial x \partial t} + \eta
\end{equation*}
with Dirichlet conditions $u(t,0)=u(t,1)=0$ and initial data $u(0,x)= u_0 \in C_0((0,1))^+$ if
\begin{enumerate}[(i)]
\item u is a continuous adapted random field on $\mathbb{R}^+ \times [0,1]$ such that $u \geq 0$ almost surely.
\item $\eta$ is a random measure on $\mathbb{R}^+ \times (0,1)$ such that: 
\begin{enumerate}
\item For every $t \geq 0$, $\eta(\left\{t\right\} \times (0,1))=0$,
\item For every $t \geq 0$, $\int_0^t \int_0^1 x(1-x) \eta(\textrm{ds,dx}) < \infty$,
\item $\eta$ is adapted in the sense that for any measurable mapping $\psi$:
\begin{equation*}
\int_0^t \int_0^1 \psi(s,x) \; \eta(\textrm{ds,dx}) \; \;  \textrm{is }  \mathscr{F}_t^W -\textrm{measurable}.
\end{equation*}
\end{enumerate}
\item For every $t \geq 0$ and every $\phi \in C^{1,2}([0,t] \times [0,1])$ with $\phi(s,0)=\phi(s,1)=0$ for every $s \in [0,t]$,
\begin{equation*}
\begin{split}
\int_0^1 u(t,x) \phi(t,x) \textrm{d}x = &  \int_0^1 u(0,x)\phi(0,x) \textrm{d}x + \int_0^t \int_0^1 u(s,x)\frac{\partial^2 \phi}{\partial x^2}(s,x) \textrm{d}x \textrm{d}s \\ & + \int_0^t \int_0^1 u(s,x) \frac{\partial \phi}{\partial t}(s,x) \; \textrm{d}x \textrm{d}s + \int_0^t \int_0^1 f(x,u(s,x))\phi(s,x) \textrm{d}x \textrm{d}s \\ & + \int_0^t \int_0^1 \phi(s,x) \sigma(x,u(s,x)) W(\textrm{d}s,\textrm{d}x) + \int_0^t \int_0^1 \phi(s,x) \; \eta(\textrm{d}s,\textrm{d}x)
\end{split}
\end{equation*}
almost surely.
\item $\int_0^{\infty} \int_0^1 u(t,x) \; \eta(\textrm{d}t,\textrm{d}x)=0$.
\end{enumerate}
\end{defn}

We now state the main result of the paper, which states that reflected SPDEs of the form (\ref{ref_SHE}) have invariant measures.

\begin{thm}\label{main}
Suppose that $f$ and $\sigma$ satisfy the conditions (I)-(IV). There exists an invariant probability measure for the reflected stochastic heat equation (\ref{ref_SHE}).
\end{thm}

\section{An $L^p$ bound for solutions to reflected SPDEs}

The aim of this section is to prove the following Theorem.

\begin{thm}\label{Lp_Bound_Main}
Suppose that $(u,\eta)$ solves the reflected stochastic heat equation (\ref{ref_SHE}). Assume that the drift and volatility functions, $f$, $\sigma$ are bounded. Then we have that, for any $\alpha >0$ and $p \geq 1$, 
\begin{equation}\label{Lp_bound}
\sup\limits_{T>0} \mathbb{E}\left[ \sup\limits_{t \leq T} \sup\limits_{x \in [0,1]} |u(t,x)e^{-\alpha (T-t)} |^p \right] < \infty.
\end{equation}
\end{thm}
Such a bound will later enable us to obtain uniform H\"{o}lder-type estimates for the functions $u(t,\cdot)$. The first step towards obtaining this bound is understanding the equation satisfied by $\tilde{u}(t,x):= e^{-\alpha(T-t)} u(t,x)$.

\begin{prop}\label{exp_SPDE}
Let u solve the reflected SPDE (\ref{ref_SHE}). Let $\tilde{u}(t,x):= e^{-\alpha(T-t)}u(t,x)$ for some $\alpha, T>0$. Then $\tilde{u}$ solves the reflected SPDE
\begin{equation}\label{dec_PDE}
\frac{\partial \tilde{u}}{\partial t}= \Delta \tilde{u} + \tilde{f}(t,x,\tilde{u}(t,x)) + \tilde{\sigma}(t,x,\tilde{u}(t,x)) \frac{\partial^2 W}{\partial x \partial t} + \tilde{\eta},
\end{equation}
where 
\begin{enumerate}
\item $\tilde{f}(t,x,z)=e^{-\alpha (T-t)}f(x,e^{\alpha(T-t)}z)+ \alpha z.$
\item $\tilde{\sigma}(t,x,z)= e^{-\alpha (T-t)}\sigma(x,e^{\alpha(T-t)}z).$
\item $\tilde{\eta}(\textrm{d}x, \textrm{d}t) = e^{-\alpha(T-t)} \eta(\textrm{d}x,\textrm{d}x).$
\end{enumerate}
\end{prop}
\begin{proof}
This can be shown by testing the equation and a change of variables.
\end{proof}

We now present some estimates for the heat kernel. We will then be able to bound the solutions to our SPDEs by first writing them in mild form and then applying these estimates, together with Burkholder's inequality and H\"{o}lder's inequality.

\begin{prop}\label{Heat}
Let $G$ denote the Dirichlet heat kernel on $[0,1]$. The following estimate holds:
\begin{equation*}
\sup\limits_{x \in [0,1]}\int_0^{\infty} \int_0^1 G(s,x,y) \textrm{d}y \textrm{d}s < \infty.
\end{equation*}
\begin{proof}
We have the following expression for $G$
\begin{equation*}
G(s,x,y)= 2\sum\limits_{k =1}^{\infty} e^{-k^2 \pi^2 s} \sin(k \pi x) \sin(k \pi y).
\end{equation*}
Calculating, we have that 
\begin{equation*}
\begin{split}
\int_0^{\infty}  &  \int_0^1 G(t,x,y) \; \textrm{d}y \textrm{d}t = 2\left| \int_0^{\infty} \int_0^1 \sum\limits_{k =1}^{\infty} e^{-k^2 \pi^2 t} \sin(k \pi x) \sin(k \pi y) \; \textrm{d}y \textrm{d}t \right| \\ & \leq 2\int_0^{\infty} \int_0^1 \sum\limits_{k=1}^{\infty} e^{-k^2 \pi^2 t} |\sin(k \pi y) | \; \textrm{d}y \textrm{d}t = 2\sum\limits_{k=1}^{\infty} \int_0^{\infty} e^{-k^2 \pi^2 t} \left( \int_0^1 |\sin (k \pi y) | \textrm{d}y \right) \textrm{d}t \\ & \leq 2\sum\limits_{k=1}^{\infty} \int_0^{\infty} e^{-k^2 \pi^2 t} \textrm{d}t = \sum\limits_{k=1}^{\infty} \frac{2}{k^2 \pi^2}= \frac{1}{3}.
\end{split}
\end{equation*}
\end{proof}
\end{prop}

\begin{prop}\label{Heat_Bound}
The following estimates hold for $p>4$
\begin{enumerate}
\item For every $0 \leq s \leq t$ such that $|t-s| \leq 1$
\begin{equation*}
\sup\limits_{x \geq 0} \left[ \int_s^t \left( \int_0^1 G(t-r,x,z)^2 \textrm{d}z \right)^{p/(p-2)} \textrm{d}r \right]^{(p-2)/2} \leq C_p |t-s|^{p-4/4}.
\end{equation*}
\item For every $0 \leq s \leq t$ such that $|t-s| \leq 1$
\begin{equation*}
\sup\limits_{x \geq 0} \left[ \int_0^s \left( \int_0^1 (G(t-r,x,z)-G(s-r,x,z)^2 \textrm{d}z \right)^{p/(p-2)} \textrm{d}r \right]^{(p-2)/2} \leq C_p |t-s|^{p-4/4}.
\end{equation*}
\item For every $x, y \in [0,1]$
\begin{equation*}
\sup\limits_{t \geq 0} \left[ \int_0^t \left( \int_0^1 (G(t-r,x,z)-G(t-r,y,z)^2 \textrm{d}z \right)^{p/(p-2)} \textrm{d}r \right]^{(p-2)/2} \leq C_p |x-y|^{p-4/2}.
\end{equation*}
\end{enumerate}
\end{prop}
\begin{proof}
We note that
\begin{equation*}
G(t,x,y)= \frac{1}{\sqrt{4 \pi t}}\sum\limits_{n = - \infty}^{n = \infty} \left[ \exp \left( - \frac{(x-y+2n)^2}{4t}\right)- \exp \left( - \frac{(x+y+2n)^2}{4t}\right) \right].
\end{equation*}
We can write this as
\begin{equation*}
\frac{1}{\sqrt{4 \pi t}} \left[ \exp \left( - \frac{(x-y)^2}{4t}\right)- \exp \left( - \frac{(x-y)^2}{4t}\right)- \exp \left( - \frac{(x+y-2)^2}{4t}\right) \right] + L(t,x,y),
\end{equation*}
where $L$ is a smooth function which vanishes at $t=0$.
To control the contributions of the first three terms, see the proof of Proposition A.1 and Proposition A.4 in \cite{Paper} for details. Note that the constants will not depend on $t$ for this case. The residual component $L$ can also be controlled by differentiating under the sum.
\end{proof}

Equipped with these heat kernel estimates, we can now prove the following bound on the white noise term which will appear in the mild form for $\tilde{u}$.

\begin{prop}\label{stoch_bound}
Suppose that $\sigma$ is bounded and $\alpha>0$. Define for $t \leq T$, $x \in [0,1]$
\begin{equation*}
I_2^T(t,x):= \int_0^t \int_0^1 e^{-\alpha(T-s)}G(t-s,x,y)\sigma(y,u(s,y)) W(\textrm{d} y,\textrm{d}s).
\end{equation*}
Then for $p \geq 1$ and $\gamma \in (0,1)$,
\begin{equation*}
\sup\limits_{T>0} \mathbb{E}\left[ \sup\limits_{s,t \in [0,T], s \neq t} \sup\limits_{x, y \in [0,1], x \neq y} \left(\frac{ |I_2^T(t,x)- I_2^T(s,y)|}{|t-s|^{\gamma /4} + |x-y|^{\gamma /2}}\right)^p \;  \right] < \infty.
\end{equation*}
\end{prop}
\begin{proof}
Fix $n \in \mathbb{N}$ such that $n \leq \lfloor T \rfloor$. Let $t,s \in [n,n+1] \cap [0,T]$ and $x,y \in [0,1]$. Assume without loss of generality that $s \leq t$. We have that 
\begin{equation*}
\begin{split}
\mathbb{E} & \left[ |I_2^T(t,x)-I_2^T(s,y)|^p \right] \leq C_p \mathbb{E} \left[ \left| \int_s^t \int_0^1 e^{-\alpha(T-r)}G(t-r,x,z)\sigma(z,u(r,y)) W(\textrm{d} z,\textrm{d}r)\right|^p \right] \\& + C_p \mathbb{E} \left[ \left| \int_0^s \int_0^1 e^{-\alpha(T-r)}\sigma(z,u(r,z))(G(t-r,x,z)-G(s-r,x,z)) W(\textrm{d} z,\textrm{d}r)\right|^p \right] \\& + C_p \mathbb{E} \left[ \left| \int_0^s \int_0^1 e^{-\alpha(T-r)}\sigma(z,u(r,z))(G(s-r,x,z)-G(s-r,y,z)) W(\textrm{d} z,\textrm{d}r)\right|^p \right]. 
\end{split}
\end{equation*}
Applying Burkholder's inequality to each of these terms allows us to bound the right hand side by
\begin{equation*}
\begin{split}
 C_p &  \mathbb{E} \left[ \left| \int_s^t \int_0^1 e^{-2\alpha(T-r)}G(t-r,x,z)^2 \sigma^2(z,u(r,y)) \textrm{d}z \textrm{d}r \right|^{p/2} \;  \right] \\& + C_p \mathbb{E} \left[ \left| \int_0^s \int_0^1 e^{-2\alpha(T-r)}(G(t-r,x,z)-G(s-r,x,z))^2 \sigma^2(z,u(r,z)) \textrm{d}z \textrm{d}r\right|^{p/2} \;  \right] \\& + C_p \mathbb{E} \left[ \left| \int_0^s \int_0^1 e^{-2\alpha(T-r)}(G(s-r,x,z)-G(s-r,y,z))^2 \sigma^2(z,u(r,z)) \textrm{d}z \textrm{d}r\right|^{p/2} \;  \right]. 
\end{split}
\end{equation*}
We focus on the first of these terms and note that the arguments for the other two are essentially the same, the difference being in which inequality from Proposition \ref{Heat_Bound} we apply. Since $\sigma$ is bounded and $s,t \in [n,n+1]$, we have that $e^{-\alpha(T-r)} \sigma^2(z,u(r,z)) \leq \|\sigma\|_{\infty}^2 e^{-\alpha(T-(n+1))}$ for $(r,z) \in [0,t] \times [0,1]$. This gives that:
\begin{equation*}
\begin{split}
\mathbb{E} & \left[ \left| \int_s^t \int_0^1 e^{-2\alpha(T-r)}G^2(t-r,x,z)\sigma^2(z,u(r,y)) \textrm{d}z \textrm{d}r \right|^{p/2} \right] \\ & \leq C_{\sigma}e^{-\alpha p (T-(n+1))/2}  \left| \int_s^t \int_0^1 e^{-\alpha(T-r)}G^2(t-r,x,z)\textrm{d}z \textrm{d}r \right|^{p/2}. 
\end{split}
\end{equation*}
By H\"{o}lder's inequality we have that
\begin{equation*}
\begin{split}
& \left| \int_s^t \int_0^1 e^{-\alpha(T-r)}G(t-r,x,z)^2\textrm{d}z \textrm{d}r \right|^{p/2} 
\\ & \leq \int_s^t e^{-\alpha p (T-r)/2} \textrm{d}r \times \left[ \int_s^t \left( \int_0^1 G(t-r,x,z)^2 \textrm{d}z \right)^{p/(p-2)} \textrm{d}r \right]^{(p-2)/2} \\ & \leq \int_0^{\infty} e^{-\alpha p t/2} \textrm{d}t \times \left[ \int_s^t \left( \int_0^1 G(t-r,x,z)^2 \textrm{d}z \right)^{p/(p-2)} \textrm{d}r \right]^{(p-2)/2}.
\end{split}
\end{equation*}
By inequality (1) from Proposition \ref{Heat_Bound}, this is at most $C_{p,\alpha} |t-s|^{(p-4)/2}$. Arguing similarly for the other terms, we obtain that for $t,s \in [n,n+1] \cap [0,T]$ and $x,y \in [0,1]$
\begin{equation*}
\mathbb{E}\left[ |I_2^T(t,x)-I_2^T(s,y)|^p \right] \leq C_{p,\sigma}e^{-\alpha p (T-(n+1))/2}\left( |t-s|^{1/2} + |x-y| \right)^{(p-4)/2}.
\end{equation*}
Let $p$ be large enough so that $\gamma < (p-10)/p$. We can then apply Corollary A.3 from \cite{Dalang} to obtain that for $t,s \in [n,n+1] \cap [0,T]$ and $x,y \in [0,1]$
\begin{equation}\label{X_n}
|I_2^T(t,x)- I_2^T(s,y)| \leq X_n( |t-s|^{\gamma/4} + |x-y|^{\gamma/2}).
\end{equation}
almost surely, where $X_n$ is a positive random variable such that 
\begin{equation}\label{Holder_Xn}
\mathbb{E}\left[ X_n^p \right] \leq C_{\gamma,p,\sigma}e^{-\alpha p (T-(n+1))/2}.
\end{equation}
Now suppose that $0 \leq s \leq t \leq T$, $x,y \in [0,1]$, and that there exists $n < m \in \mathbb{N}$ such that $s \in [n,n+1]$ and $t \in [m,m+1]$. We then have that
\begin{equation}\label{I_2_Hold}
\begin{split}
|I_2^T(t,x)- I_2^T(s,y)|  \leq & |I_2^T(t,x) - I_2^T(t,y)| + |I_2^T(t,y) - I_2^T(m,y)| 
\\ & + \left( \sum\limits_{i = n +2}^m |I_2^T(i, y) - I_2^T(i-1,y) | \right) + |I_2^T(n+1,y) - I_2^T(s,y)|,
\end{split}
\end{equation}
where we use the convention that the sum is zero if $m = n+1$.
By applying (\ref{X_n}), we then obtain that (\ref{I_2_Hold}) is at most
\begin{equation}\label{I_2_Hold_2}
\begin{split}
 X_m |x-y|^{\gamma/2} + X_m |t-m|^{\gamma/4}+ \left( \sum\limits_{i = n +2}^m X_{i-1} \right) + X_n |(n+1) -s|^{\gamma/4}.
\end{split}
\end{equation}
Since $n <m$, we have that $|t-m| \leq |t-s|$ and $|(n+1) - s| \leq |t-s|$. In addition, if $m \geq n +2$, we have that $|t -s| \geq 1$. Therefore, (\ref{I_2_Hold_2}) can be bounded by
\begin{equation*}
\left(\sum\limits_{i = n}^m X_i \right) \times ( |t-s|^{\gamma/4} + |x-y|^{\gamma/2}) \leq \left( \sum\limits_{i =0}^{\lfloor T \rfloor} X_i \right) (|t-s|^{\gamma/4} + |x-y|^{\gamma/2} ).
\end{equation*}
Altogether we have shown that for any $s,t \in [0,T]$ and $x,y \in [0,1]$,
\begin{equation}\label{Holder_I2}
|I_2^T(t,x)- I_2^T(s,y)| \leq \left( \sum\limits_{i =0}^{\lfloor T \rfloor} X_i \right) (|t-s|^{\gamma/4} + |x-y|^{\gamma/2} ).
\end{equation}
Let $X:= \left( \sum\limits_{i =0}^{\lfloor T \rfloor} X_i \right)$. Calculating gives that
\begin{equation*}
\mathbb{E}\left[ X^P \right] = \mathbb{E}\left[ \left(\sum\limits_{n=0}^{\lfloor T \rfloor} X_n \right)^p \; \right] \leq \mathbb{E}\left[ \left( \sum\limits_{n=0}^{\lfloor T \rfloor} e^{\alpha p (T-(n+1))/4}X_n^p \right)  \right] \times \left( \sum\limits_{n=0}^{\lfloor T \rfloor} e^{-\alpha q (T-(n+1))/4} \right)^{p/q},
\end{equation*}
where $q=p/(p-1)$. By (\ref{Holder_Xn}), we obtain that this is at most
\begin{equation*}
C_{\gamma,p, \sigma} \left(\sum\limits_{n=0}^{\infty} e^{-\alpha p (n+1)/4}\right) \times \left( \sum\limits_{n=0}^{\infty} e^{-\alpha q (n+1)/4} \right)^{p/q}= C_{\gamma,p,\sigma,\alpha}<\infty.
\end{equation*}
Importantly, this is independent of $T$. By (\ref{Holder_I2}), we then have that 
\begin{equation*}
\mathbb{E}\left[  \sup\limits_{s,t \in [0,T], s \neq t} \sup\limits_{x, y \in [0,1], x \neq y}  \left(\frac{ |I_2^T(t,x)- I_2^T(s,y)|}{|t-s|^{\gamma /4} + |x-y|^{\gamma /2}}\right)^p \;  \right] \leq C_{\gamma,p,\sigma,\alpha}.
\end{equation*}
Taking the supremum over $T>0$ concludes the proof.
\end{proof}

\begin{cor}\label{infinity}
Let $I_2^T$ be as in Proposition \ref{stoch_bound}. For $p \geq 1$, we have
\begin{equation*}
\sup\limits_{T>0} \mathbb{E}\left[ \sup\limits_{t \in [0,T]} \sup\limits_{x \in [0,1]} |I_2^T(t,x)|^p \right] \leq C_{p,\sigma,\alpha}.
\end{equation*}
\end{cor}
\begin{proof}
Note that $I_2^T(t,0)=0$ for every $t \in [0,T]$ almost surely. Therefore, for $x \in (0,1]$,
\begin{equation*}
\begin{split}
|I_2^T(t,x)|= |I_2^T(t,x) - I_2^T(t,0)| & \leq \frac{|I_2^T(t,x) - I_2^T(t,0)|}{x^{1/4}} 
\\ & \leq \sup\limits_{s,t \in [0,T], s \neq t} \sup\limits_{x, y \in [0,1], x \neq y} \frac{ |I_2^T(t,x)- I_2^T(s,y)|}{|t-s|^{1/8} + |x-y|^{1/4}}.
\end{split}
\end{equation*}
By taking the supremum on the left hand side and then taking the $L^p(\Omega)$-norm, we see that the result follows from Proposition \ref{stoch_bound}.
\end{proof}

We are now in position to prove the main Theorem for this section. Throughout the proof, we will denote the infinity norm on $[0,T] \times [0,1]$ by $\| \cdot \|_{\infty, T}$. That is, for $f: [0,T] \times [0,1] \rightarrow \mathbb{R}$, we define
\begin{equation*}
\| f \|_{\infty, T} := \sup\limits_{ t \in [0,T]} \sup\limits_{ x \in [0,1]} |f(t,x)|.
\end{equation*}

\begin{proof}[Proof of Theorem \ref{Lp_Bound_Main}]
Fix $T>0$. Let $\tilde{u}$ be as in Proposition \ref{exp_SPDE}. Then $\tilde{u}$ solves the reflected SPDE (\ref{dec_PDE}). Let $\tilde{v}(t,x)$ solve the SPDE 
\begin{equation*}
\frac{\partial \tilde{v}}{\partial t}= \Delta \tilde{v} + \tilde{f}(t,x,\tilde{u}(t,x)) + \tilde{\sigma}(t,x,\tilde{u}(t,x)) \frac{\partial^2 W}{\partial x \partial t}.
\end{equation*}
We then have, by Theorem 1.4 in \cite{NP}, that $\|\tilde{u} \|_{\infty,T} \leq 2 \| \tilde{v} \|_{\infty,T}$ almost surely. Writing $\tilde{v}$ in mild form gives 
\begin{equation*}
\begin{split}
\tilde{v}(t,x)= & \int_0^t \int_0^1 G(t-s,x,y)\tilde{f}(s,y,\tilde{u}(s,y)) \textrm{d}y \textrm{d}s \\ & + \int_0^t \int_0^1 e^{-\alpha(T-s)}G(t-s,x,y)\sigma(s,y,\tilde{u}(s,y)) W(\textrm{d}y\textrm{d}s)=:I_1^T(t,x) + I_2^T(t,x) .
\end{split}
\end{equation*}
It follows that 
\begin{equation}\label{Lp}
\mathbb{E}\left[ \|\tilde{u}\|_{\infty,T}^p \right] \leq C_p \left( \mathbb{E}\left[ \|I_1^T\|_{\infty,T}^p \right] + \mathbb{E}\left[ \|I_2^T\|_{\infty,T}^p \right] \right).
\end{equation}
Bounding the $I_1^T$ term, we obtain by applying Proposition \ref{Heat}:
\begin{equation*}
\begin{split}
\left| I_1^T(t,x) \right| = & \left| \int_0^t \int_0^1 G(t-s,x,y) \tilde{f}(s,y,\tilde{u}(s,y)) \; \textrm{d}y \textrm{d}s \right| \\ & \leq \alpha \left| \int_0^t \int_0^1 G(t-s,x,y) \tilde{u}(s,y)  \; \textrm{d}y \textrm{d}s \right| + C_f \left| \int_0^t \int_0^1 G(t-s,x,y)  \; \textrm{d}y \textrm{d}s \right| \\ & \leq \frac{\alpha \|\tilde{u}\|_{\infty,T} + C_f}{3}.
\end{split}
\end{equation*}
This gives that 
\begin{equation*}
\mathbb{E}\left[ \|I_1^T\|_{\infty,T}^p \right] \leq C_p \frac{\alpha^p \mathbb{E}\left[ \|\tilde{u}\|_{\infty,T}^p \right] }{3^p} + C_p \frac{C_f^p}{3^p}.
\end{equation*}
Applying Corollary \ref{infinity} and using the inequality (\ref{Lp}), we obtain that 
\begin{equation}\label{r}
\mathbb{E}\left[ \|\tilde{u}\|_{\infty,T}^p \right] \leq  C_{p,\sigma,\alpha,f} + \frac{\alpha^p}{3^p}C_p \; \mathbb{E}\left[ \|\tilde{u}\|_{\infty,T}^p \right].
\end{equation}
Choosing $\alpha$ to be sufficiently small, noting that the result for larger values of $\alpha$ follows from the result for smaller $\alpha$, simple rearrangement of (\ref{r}) gives that 
\begin{equation*}
\mathbb{E}\left[ \|\tilde{u}\|_{\infty,T}^p \right] \leq \tilde{C}_{p,\sigma,\alpha,f}. 
\end{equation*}
Since this bound is independent of $T$, we have the result.
\end{proof}

\section{Tightness of the Sequence $(u(t,\cdot))_{t \geq 0}$ and Proof of Theorem \ref{main}}

We recall that, by Arzela-Ascoli Theorem, the relatively compact sets in $C_0(0,1)$ are those which are equicontinuous. It follows that collections of functions for which we can uniformly bound some H\"{o}lder norm are relatively compact. Therefore, in order to prove tightness of $(u(t,\cdot))_{t \geq 0}$, it is enough to show that
\begin{equation*}
\sup\limits_{T \geq 0} \mathbb{E}\left[ \sup\limits_{x,y \in [0,1],x \neq y} \frac{|u(T,x)-u(T,y)|}{|x-y|^{\alpha}} \right] < \infty
\end{equation*}
for some $\alpha>0$. To show this, we use estimate (\ref{Lp_bound}) and follow the work of Dalang and Zhang in \cite{Dalang2}, in which the authors prove H\"{o}lder continuity for reflected SDPEs. Since the supremum over $T$ appears outside the expectation, we can apply the reasoning from Zhang \cite{Dalang2} to $\tilde{u}(t,x)=e^{-\alpha(T-t)}u(t,x)$ for each $T>0$, and then take the supremum over $T$.

\begin{proof}[Proof of Theorem \ref{main}]
Let $T>0$ and define $\tilde{u}(t,x):= e^{-\alpha(T-t)}u(t,x)$. By Proposition \ref{exp_SPDE}, we have that $(\tilde{u}, \tilde{\eta})$ solves the reflected SPDE (\ref{dec_PDE}). Define $\tilde{v}$ as the solution to the SPDE
\begin{equation*}
\frac{\partial \tilde{v}}{\partial t}= \Delta \tilde{v} + \tilde{f}(t,x,\tilde{u}(t,x)) + \tilde{\sigma}(t,x,\tilde{u}(t,x)) \frac{\partial^2 W}{\partial x \partial t},
\end{equation*}
with Dirichlet boundary conditions $v(t,0)=v(t,1)=0$ and zero initial data. We now examine the H\"{o}lder continuity of $\tilde{v}$. Writing $\tilde{v}$ in mild form, we have that
\begin{equation*}
\begin{split}
\tilde{v}(t,x)= & \int_0^t \int_0^1 G(t-s,x,y)\tilde{f}(s,y,\tilde{u}(s,y)) \textrm{d}y \textrm{d}s \\ & + \int_0^t \int_0^1 G(t-s,x,y)\tilde{\sigma}(s,y,\tilde{u}(s,y)) W(\textrm{d}y, \textrm{d}s)=: I_1^T(t,x)+ I_2^T(t,x).
\end{split}
\end{equation*}
Fix $\gamma \in (0,1)$. By Proposition \ref{stoch_bound}, we have that for $p \geq 1$ there exists $X^T \in L^p(\Omega)$ such that, for every $s,t \in [0,T]$ and every $x,y \in [0,1]$
\begin{equation}\label{I2H}
|I_2^T(t,x)-I_2^T(s,y)| \leq X^T ( |t-s|^{\gamma/4} + |x-y|^{\gamma/2}) 
\end{equation}
almost surely, with the following uniform bound on the $X^T$:
\begin{equation*}
\sup\limits_{T>0} \mathbb{E}\left[ |X^T|^p \right]< \infty. 
\end{equation*}
We now control the H\"{o}lder norm of $I_1^T$. We have that
\begin{equation*}
\begin{split}
|I_1^T(t,x)-I_1^T(s,y)| \leq & \left |\int_s^t \int_0^1 G(t-r,x,z)\tilde{f}(r,z,\tilde{u}(r,z)) \textrm{d}z \textrm{d}r \right| \\ & + \left |\int_0^s \int_0^1 (G(t-r,x,z)-G(s-r,x,z))\tilde{f}(r,z,\tilde{u}(r,z)) \textrm{d}z \textrm{d}r \right| \\ & + \left |\int_0^s \int_0^1 (G(s-r,x,z)-G(s-r,y,z))\tilde{f}(r,z,\tilde{u}(r,z)) \textrm{d}z \textrm{d}r \right|.
\end{split}
\end{equation*}
For the first of these terms, we note that for $n \in \mathbb{N}$ and $\tilde{s}, \tilde{t} \in [n,n+1]$, we have that 
\begin{equation}\label{rand}
\begin{split}
& \left |\int_{\tilde{s}}^{\tilde{t}} \int_0^1 G(\tilde{t}-r,x,z)\tilde{f}(r,z,\tilde{u}(r,z)) \textrm{d}z \textrm{d}r \right|  \\ & \leq  (C_f + \alpha) \left |\int_{\tilde{s}}^{\tilde{t}} \int_0^1 G(\tilde{t}-r,x,z) \textrm{d}z \textrm{d}r \right| e^{-\alpha/2 (T-(n+1))}\sup\limits_{ r \in [0,T] } \sup\limits_{z \in [0,1]} \left[ 1+ e^{\alpha /2 (T-r)}\tilde{u}(r,z) \right],
\end{split}
\end{equation}
where we make use of the bound 
\begin{equation*}
\begin{split}
|\tilde{f}(r,z,\tilde{u}(r,x)| & \leq C_fe^{-\alpha(T-r)} + \alpha \tilde{u}(r,x) \\ & \leq (C_f + \alpha)e^{-\alpha/2 (T-r)}( 1+ e^{\alpha /2 (T-r)}\tilde{u}(r,z)).
\end{split}
\end{equation*}
This then gives that, for $p$ large enough so that $(p-4)/p > \gamma$, 
\begin{equation*}
\begin{split}
& \left |\int_{\tilde{s}}^{\tilde{t}} \int_0^1 G(\tilde{t}-r,x,z)\tilde{f}(r,z,\tilde{u}(r,z)) \textrm{d}z \textrm{d}r \right| \leq   C_{f,\alpha} Y^T e^{-\alpha/2 (T-(n+1))}\left |\int_{\tilde{s}}^{\tilde{t}} \int_0^1 G(\tilde{t}-r,x,z) \textrm{d}z \textrm{d}r \right| \\ & \leq  C_{f,\alpha} Y^T e^{-\alpha/2 (T-(n+1))}\left( \int_{\tilde{s}}^{\tilde{t}} \left[\int_0^1 G(\tilde{t}-r,x,z) \textrm{d}z \right]^{2p/(p-2)}\textrm{d}r  \right)^{(p-2)/2p} \\ &  \leq  C_{f,\alpha} Y^T e^{-\alpha/2 (T-(n+1))} \left( \int_{\tilde{s}}^{\tilde{t}} \left[\int_0^1 G(\tilde{t}-r,x,z)^2 \textrm{d}z \right]^{p/(p-2)}\textrm{d}r \right)^{(p-2)/2p} \\ & \leq C_{f, \alpha, p} Y^T e^{-\alpha/2 (T-(n+1))}  |\tilde{t} - \tilde{s}|^{p-4/4p} \\ & \leq C_{f, \alpha, p} Y^T e^{-\alpha/2 (T-(n+1))}  |\tilde{t} - \tilde{s}|^{\gamma/4},
\end{split}
\end{equation*}
where $\sup\limits_{T \geq 0} \mathbb{E}\left[ (Y^T)^p \right] < \infty$. For a general $s,t \in [0,T]$ we then have that 
\begin{equation*}
\begin{split}
\left |\int_s^t \int_0^1 G(t-r,x,z)\tilde{f}(r,z,\tilde{u}(r,z)) \textrm{d}z \textrm{d}r \right| \leq & C_{f, \alpha, p} Y^T\left[ \sum\limits_{n=0}^{\lfloor T \rfloor} e^{-\alpha/2 (T-(n+1))}   \right]  |t-s|^{\gamma/4} 
\\ \leq & C_{f, \alpha, p} Y^T\left[ \sum\limits_{n=0}^{\infty} e^{-\alpha/2 (n+1)}   \right]  |t-s|^{\gamma/4}  
\\ = & C_{f, \alpha, p} Y^T|t-s|^{\gamma/4}.
\end{split}
\end{equation*}
 Arguing in the same way and applying the other estimates from Proposition \ref{Heat_Bound}, we obtain that for every $t,s \in [0,T]$ and every $x,y \in [0,1]$
\begin{equation}\label{I1H}
|I_1^T(t,x)-I_1^T(s,y)| \leq C_{f, \alpha, p} Y^T ( |t-s|^{\gamma/4} + |x-y|^{\gamma/2}) 
\end{equation}
almost surely, where $\sup\limits_{T>0}\mathbb{E}\left[ |Y^T|^p \right]<\infty$. Setting $Z^T= X^T +C_{f, \alpha, p} Y^T$, we have from (\ref{I2H}) and (\ref{I1H}) that $Z^T$ bounds the $(\gamma/4, \gamma/2)$-H\"{o}lder norm of $\tilde{v}$. In the proof of Theorem 3.3 in \cite{Dalang2} and Theorem 3.16 in \cite{Paper}, it is shown that the H\"{o}lder norm of $\tilde{v}$ controls the H\"{o}lder norm of $\tilde{u}$. More precisely, we have that 
\begin{equation*}
|\tilde{u}(t,x)- \tilde{u}(s,y)| \leq C_{\gamma} Z^T ( |t-s|^{\gamma/4} + |x-y|^{\gamma/2})
\end{equation*}
for every $t,s \in [0,T]$ and every $x,y \in [0,1]$ almost surely. In particular, we obtain
\begin{equation*}
\mathbb{E}\left[ \sup\limits_{x,y \in [0,1],x \neq y} \frac{|u(T,x)-u(T,y)|^p}{|x-y|^{\gamma p /2}} \right]=  \mathbb{E}\left[ \sup\limits_{x,y \in [0,1],x \neq y} \frac{|\tilde{u}(T,x)-\tilde{u}(T,y)|^p}{|x-y|^{\gamma p/2}} \right] \leq C_{\gamma}^p \mathbb{E}\left[|Z^T|^p \right].
\end{equation*}
Noting that $\sup\limits_{T>0} \mathbb{E}\left[ |Z^T|^p \right] < \infty$ concludes the proof.
\end{proof}

\textbf{Acknowledgements.} 
The author would like to thank Ben Hambly for his guidance and discussions, and Imanol P\'{e}rez for his careful proofreading. This research was supported by EPSRC (EP/L015811/1).

\end{document}